\documentclass[a4paper,twoside,11pt]{article}
\usepackage{a4wide}
\usepackage[utf8]{inputenc}
\usepackage[english]{babel}
\usepackage{color} 
\usepackage{amsmath,amssymb,amsthm,amsfonts}
\usepackage{bbm}
\usepackage{graphicx}
\usepackage{hyperref}
\usepackage[d]{esvect}
\usepackage{abstract}
\usepackage{xparse}

\usepackage{enumitem}
\newlist{hyp}{enumerate}{1}
\setlist[hyp]{label=(H\arabic*),leftmargin=15mm,labelsep=3mm}

\theoremstyle{plain}
\newtheorem{theorem}{Theorem}
\newtheorem*{theorem*}{Theorem}
\newtheorem{lemma}[theorem]{Lemma}
\newtheorem{corollary}[theorem]{Corollary}
\newtheorem{proposition}[theorem]{Proposition}
\newtheorem{fact}[theorem]{Fact}

\theoremstyle{definition}

\newtheorem*{definition*}{Definition}

\theoremstyle{remark}

\numberwithin{equation}{section}

\def\N{\ensuremath{\mathbb{N}}}

\def\R{\ensuremath{\mathbb{R}}}
\def\Z{\ensuremath{\mathbb{Z}}}

\def\ep{\varepsilon}

\def\E{\ensuremath{\mathbf{E}}}

\def\P{\ensuremath{\mathbf{P}}}

\def\Ind{\ensuremath{\mathbbm{1}}}

\def\to{\rightarrow}

\def\tand{\ensuremath{\text{ and }}}
\def\tif{\ensuremath{\text{ if }}}
\def\tas{\ensuremath{\text{ as }}}
\def\ton{\ensuremath{\text{ on }}}

\def\widebar{\overline} 

\def\Smax{M^S}
\def\SLmax{\widetilde M^S}
\def\Xmax{M^X}
\def\Sabsmax{M^{|S|}}
\def\SLabsmax{\widetilde M^{|S|}}
\def\Xabsmax{M^{|X|}}
\def\DeltaSX{\Delta^{SX}}

\def\gammabar{\widebar{\gamma}}

\DeclareDocumentCommand \TODO { o }%
{%
\IfNoValueTF {#1}%
{{\color{red}TODO}}%
{{\color{red}TODO:#1}}%
}

\author{\textsc{Pascal Maillard}\thanks{
Département de Mathématiques, Université Paris-Sud, Université Paris-Saclay, 91405 Orsay Cedex, France. e-mail: pascal dot maillard at u-psud dot fr. Partially supported by Grant ANR-14-CE25-0014 (ANR GRAAL).}}
\title{The maximum of a tree-indexed random walk in the big jump domain}
\usepackage[normalem]{ulem}
\date{May 19, 2015}
\begin{document}

\maketitle

\begin{abstract}
 We consider random walks indexed by arbitrary finite random or deterministic trees. We derive a simple sufficient criterion which ensures that the maximal displacement of the tree-indexed random walk is determined by a single large jump. This criterion is given in terms of four quantities : the tail and the expectation of the random walk steps, the height of the tree and the number of its vertices. The results are applied to critical Galton--Watson trees with offspring distributions in the domain of attraction of a stable law.
\end{abstract}


\section{Introduction}
Let $S_n = X_1+\cdots+X_n$ be a symmetric random walk with $\P(X_1 > x) \sim x^{-\alpha}$ as $x\to\infty$, for some $\alpha>0$ (here and throughout, we write $a_n \sim b_n$ if $a_n/b_n\to 1$ for two sequences $(a_n)_{n\ge0}$ and $(b_n)_{n\ge0}$ of positive numbers). It is well-known and easy to show that $\Xmax_n = \max(X_1,\ldots,X_n) \asymp n^{1/\alpha}$ as $n\to\infty$. On the other hand, standard random walk theory (see, e.g. \cite{Feller1971}) gives that $S_n = n^{1/(2\wedge\alpha) + o(1)}$ as $n\to\infty$. In other words, $S_n$ and $\Xmax_n$ are (roughly) of the same order if and only if $\alpha \le 2$. 

Now consider a critical, finite variance branching random walk, whose random walk steps are distributed as above. This means that starting from a Galton--Watson tree with offspring distribution of mean 1 and finite variance, we assign iid random variables $X_v$ to each non-root vertex $v$, distributed as above, and let $S_v$ be the sum over all $X_u$ where $u$ runs through all the non-root vertices on the path from the root to $v$. Condition the tree on having $n$ non-root vertices and let $\Xmax_n$ and $\Smax_n$ be the maximum over all $X_v$ and $S_v$, respectively. Of course, $\Xmax_n \asymp n^{1/\alpha}$ as above. As for $\Smax_n$, Kesten~\cite{Kesten1995} proved that if $\alpha > 4$, then $\Smax_n \asymp n^{1/4}$. Subsequently, Janson and Marckert~\cite{Janson2005} showed that in general, $\Smax_n \asymp n^{1/(4\wedge \alpha)}$, so that $\Smax_n \asymp \Xmax_n$ if and only if $\alpha \le 4$.


It is well-known that a critical, finite variance Galton--Watson tree conditioned on having $n$ vertices converges after rescaling to Aldous' \emph{continuum random tree} \cite{Aldous1993}, a random metric space of Hausdorff dimension 2. 
Looking at the above results, one question immediately comes to mind: for a random walk indexed by a large ``$D$-dimensional'' random tree, a notion to be made precise, is it true that $\Smax \asymp \Xmax$ if and only if $\alpha \le 2D$, with $\Smax= \max_v S_v$ and $\Xmax= \max_v X_v$? In this article, we provide a partial response to this question. 

The meaning we give to the ``dimension'' of a tree is very crude and simple: we say that a tree is of dimension (at least) $D>1$ if its height is at most of order $V^{1/D}$, where $V$ is the number of its vertices. We will make this definition precise in two ways, yielding two different settings:
\begin{enumerate}
 \item through a growing sequence of (possibly random) trees of height $H_n$ and number of vertices $V_n$, satisfying for each $\ep>0$, $H_n \le V_n^{1/D+\ep}$ with high probability, and
 \item through a single random tree of height $H$ and number of vertices $V$ satisfying for each $\ep>0$, $\P(H\le V^{1/D+\ep},\,V\ge n) \le C_\ep n^{-\kappa}$ for some large enough $\kappa$.
\end{enumerate}
Under the condition $\alpha < 2D$ (or, $\alpha < D$ for non-centered random walk), we then prove the following:
\begin{itemize}
 \item In Setting 1, we have $\Smax_n/\Xmax_n \to 1$ in probability as $n\to\infty$, with $\Smax_n$ and $\Xmax_n$ being respectively the maximal displacement and the size of the maximal jump in the $n$-th process (Theorem~\ref{th:1})
 \item In Setting 2, we have $\P(\Smax > x) \sim \P(\Xmax > x)$ as $x\to\infty$, where $\Smax$ and $\Xmax$ are as above (Theorem~\ref{th:2})
\end{itemize}
We thus have a very easy to verify sufficient criterion for $\Smax$ to be of the same order as $\Xmax$ (in fact, they are approximately equal for reasons explained below).

We expect the notion of dimension used here to coincide with other notions in typical cases of interest. We illustrate this through the example of critical Galton--Watson trees whose offspring distribution is in the domain of attraction of a stable law (which have dimensions $D\in[2,\infty)$).

We finish this introduction by a review of the existing literature on tree-indexed random walks with heavy tails (without pretending to be exhaustive). To the knowledge of the author, these have only been considered so far only on Galton--Watson trees, under the name of branching random walks. For critical, finite-variance Galton--Watson trees, the results of our Theorem~\ref{th:1} were shown by Janson and Marckert \cite{Janson2005}\footnote{Much more is known for critical, finite-variance Galton--Watson trees: under the condition $\P(|X|>x) = o(x^{-4})$, Janson and Marckert~\cite{Janson2005} showed that a certain exploration process of the branching random walk, called the \emph{discrete snake}, converges after renormalization uniformly to a continuous process called the \emph{Brownian snake} introduced by Le~Gall~\cite{LeGall1993}. This result was first proven for $\alpha > 8$ by Marckert and Mokkadem~\cite{Marckert2003}. On the other hand, for $\alpha \le 4$, Janson and Marckert showed convergence of the discrete snake w.r.t.\ a certain topology similar to Skorokhod's $M_1$-topology to a certain non-continuous process called the \emph{jumping snake} or \emph{hairy snake}.}. Lalley and Shao~\cite{Lalley2013} consider symmetric stable branching L\'evy processes of index $\alpha\in(0,2)$, with critical binary offspring distribution, for which they prove the analogue of our Theorem~\ref{th:2} through analysis of a certain pseudo-differential equation. This work was an inspiration to the current article. It is easy to show that their results can be recovered from our results applied to a discrete skeleton of the branching L\'evy process.

On supercritical Galton--Watson trees (which correspond to $D=\infty$), Durrett~\cite{Durrett1983} considers the maximal displacement at generation $n$ of a tree-indexed random walk with regularly varying tails and shows that it approaches the maximal jump size until generation $n$ as $n\to\infty$. This result is easily recovered by our Theorem~\ref{th:1}. Bhattacharya, Hazra and Roy \cite{Bhattacharya2014} recently extended these results to the collection of extremal particles. B\'erard and Maillard~\cite{Berard2013} considered a supercritical branching random walk with regularly varying tails and with selection of the $N$ maximal particles, for large $N$. Finally, Gantert~\cite{Gantert2000} studied the maximum of supercritical branching random walk with streched exponential tails. 

To conclude, our contributions in this article are the following:
\begin{itemize}
	\item to unify previous results on the maximal displacement of heavy-tailed branching random walks,
	\item to give a simple, transparent proof of these results, and
	\item to generalize them to arbitrary trees satisfying an easily verifiable condition.
\end{itemize}

%

\section{Definitions and statements of results}

The following notation will be used throughout the article. We fix a real-valued random variable $X$, whose law is, for now, arbitrary. We say that a sequence $(b_n)_{n\in\N}$ is a \emph{natural scale sequence}\footnote{This is the same definition as in \cite{Denisov2008} apart from the fact that we require furthermore that $b_n$ is non-decreasing in $n$.} if it is non-decreasing and if the family of random variables $((X_1+\cdots+X_n)/b_n)_{n\in\N}$ is tight, where $X_1,X_2,\ldots$ are iid copies of $X$.

Let $T$ be a finite tree, deterministic or random, with root $\rho$. Denote by $\mathcal V$ the set of vertices and set $\mathcal V^* = \mathcal V\backslash\{\rho\}$. 
We denote by $V=|\mathcal V^*|$ the number of non-root vertices of the tree $T$ and by $H$ its height/depth, i.e. the largest distance between $\rho$ and another vertex.

Let $(X_v)_{v\in\mathcal V^*}$ be iid of the same law as $X$. Set
\[
 \forall v \in\mathcal V: S_v = \sum_{\rho \ne u \le v} X_u,
\]
where $u\le v$ means that $u$ lies on the path from the root to $v$ (including $v$ itself). The collection $(S_v)_{v\in\mathcal V}$ is then called the \emph{random walk indexed by the tree $T$}. 
Let $\mathcal L\subset\mathcal V$ be the subset of leaves (i.e. vertices without descendant) of the tree. We then define,
\begin{align*}
\Smax &= \max_{v\in \mathcal V} S_v, & \SLmax &= \max_{v\in\mathcal L} S_v, & \Xmax &= \max_{v\in \mathcal V^*} X_v\\
\Sabsmax &= \max_{v\in \mathcal V} |S_v|, & \SLabsmax &= \max_{v\in\mathcal L} |S_v|, & \Xabsmax &= \max_{v\in \mathcal V^*} |X_v|.
\end{align*}
Note that trivially, $\SLmax \le \Smax$ and $\SLabsmax \le \Sabsmax$. We further define
\[
 \DeltaSX = \max\left\{|\Smax - \Xmax|,|\SLmax - \Xmax|,|\Sabsmax - \Xabsmax|,|\SLabsmax - \Xabsmax|\right\}.
\]

We will also consider sequences $(T_n)_{n\in\N}$ of random trees, in which case we denote by $V_n,H_n,\Xmax_n,\Smax_n$, etc.\  the objects introduced above corresponding to the tree $T_n$.

\def\Dcrit{D_{\mathrm{crit}}}

We now introduce the assumptions on the class of tree-indexed random walks we will focus on. The assumption on the law of $X$ is
\begin{hyp}
 \item[(XR)] There exists $\alpha > 0$, such that $\P(X > x)$ and $\P(|X| > x)$ are regularly varying\footnote{The definition and basic properties of regularly varying functions are recalled in Section~\ref{sec:reg-var}.} at $\infty$ with index $-\alpha$. In this case, we set 
 \[
 \Dcrit(X) =
 \begin{cases}
  \max(1,\tfrac \alpha 2) & \tif \E[|X|]<\infty\tand \E[X] = 0\\
  \max(1,\alpha) & \text{ otherwise}.
 \end{cases}
\]
\end{hyp}

As for the underlying tree, recall from the introduction that we consider two different settings, the first involving a sequence $(T_n)_{n\in\N}$ of growing random trees, the second involving a fixed, random tree $T$. In the first setting, the assumptions on the sequence $(T_n)_{n\in\N}$ are

\begin{hyp}
 \item[(Tn1)] $V_n\to+\infty$ in probability, as $n\to\infty$. 
 \item[(Tn2)] There exists $D>1$, such that for every $\ep > 0$, 
 \[
 \P(H_n > V_n^{\frac 1 D+\ep})\to 0\quad\tas n\to\infty.
 \]
\end{hyp}

We then have the following result:

\begin{theorem}
\label{th:1}
 Assume (XR), (Tn1) and (Tn2). Assume that $D > \Dcrit(X)$, with $D$ and $\Dcrit(X)$ from (Tn2) and (XR), respectively.
 Then
 \begin{align*}
 \lim_{n\to\infty} \frac{\Smax_n}{\Xmax_n} = \lim_{n\to\infty} \frac{\SLmax_n}{\Xmax_n} = \lim_{n\to\infty} \frac{\Sabsmax_n}{\Xabsmax_n} = \lim_{n\to\infty} \frac{\SLabsmax_n}{\Xabsmax_n} = 1,\quad\text{in probability.}
 \end{align*}
 In particular, if $V_n/n\to 1$ in probability as $n\to\infty$ and if $a_n$ and $\tilde a_n$ are such that $\P(X>a_n) \sim 1/n$ and $\P(|X|>\tilde a_n) \sim 1/n$ as $n\to\infty$, then the random variables $a_n^{-1} \Smax_n$, $a_n^{-1}\SLmax_n$, $\tilde a_n^{-1}\Sabsmax_n$ and $\tilde a_n^{-1}\SLabsmax_n$ converge in law as $n\to\infty$ to a Pareto law on $[0,\infty)$ with distribution function $F(x) = \exp(-x^{-\alpha})$.
\end{theorem}

In the second setting, the assumptions on the law of the tree $T$ are the following. Here, the generating function of $V$ is denoted by $g_V(s) = \E[s^V]$.
\begin{hyp}
 \item[(T1)] There exists $\beta \in (0,1]$, such that $1-g_V(1-s)$ is regularly varying at $0$ with index~$\beta$.
 \item[(T2)] There exists $D>1$ and $\gammabar > 1\vee\alpha$, with $\alpha$ from (XR), such that for every $\ep > 0$, 
 \[
 \P(H > V^{\frac 1 D+\ep},\,V \ge n)\times n^{\beta\gammabar}\to 0\quad\tas n\to\infty.
 \]
 Here $\beta$ is the constant from (T1).
\end{hyp}
Our result under the previous hypotheses is the following:


\begin{theorem}
\label{th:2}
 Assume (XR), (T1) and (T2), Assume that $D > \Dcrit(X)$, with $D$ and $\Dcrit(X)$ from (T2) and (XR), respectively. Then, with $g_V$ from (T1),
 \begin{align*}
  &\P(\Smax > x) \sim \P(\SLmax > x) \sim \P(\Xmax > x) = 1-g_V(1-\P(X>x))&&\tas x\to\infty\\ \tand\quad&\P(\Sabsmax > x) \sim \P(\SLabsmax > x) \sim \P(\Xabsmax > x) = 1-g_V(1-\P(|X|>x))&&\tas x\to\infty.
 \end{align*}
In fact, we have
 \[
  \lim_{x\to\infty} \frac{\P(\DeltaSX> x)}{\P(\Xmax > x)} = \lim_{x\to\infty} \frac{\P(\DeltaSX> x)}{\P(\Xabsmax > x)} = 0.
 \]
\end{theorem}

The basic estimate needed in the proof of Theorems~\ref{th:1} and \ref{th:2} is the following proposition, which holds in full generality.

\begin{proposition}
\label{prop:tree_bound}
 Let $(b_n)_{n\in\N}$ be a natural scale sequence. For every $z\ge b_H$ and $y\ge 0$, we have
\[
 \P(\DeltaSX > y)
  \le \frac {HV} 2 \P(|X| > z)^2 + CV\exp(-y/z),
\]
where the constant $C$ depends only on the law of $X$ and on the sequence $(b_n)_{n\in\N}$. 
\end{proposition}
We comment in Section~\ref{sec:discussion} below on the applicability of Proposition~\ref{prop:tree_bound} and on possible improvements. For now, we just mention that the result is not at all optimal, but in the case where $X$ has regularly varying tails, which is the focus of this article, it is more than enough for our purposes.

\def\alphaGW{\alpha_{T}}

Theorems~\ref{th:1} and \ref{th:2} apply to a large class of tree-indexed random walks. We consider now the particular case of (critical) \emph{branching random walks}, i.e.~random walks indexed by (critical) Galton--Watson trees. These have been well-studied in the literature (see below for a survey of existing results). Let $\mathbf p = (p_n)_{n\in\N}$ be a probability distribution on $\N = \{0,1,2,\ldots\}$. We recall that a Galton--Watson tree $T$ with offspring distribution $\mathbf p$ is a random rooted tree, such that the degree $d_\rho$ of the root follows the law $\mathbf p$ and the subtree of the neighbors of the root are independent copies of the tree $T$, independent of $d_\rho$. We say that $T$ is \emph{critical} if $\mathbf p$ has mean 1, i.e.~$\sum n p_n = 1$. In this case $T$ is finite almost surely. The following proposition now says that Theorems~\ref{th:1} and \ref{th:2} can be applied for random walks indexed by critical Galton--Watson trees:
\begin{proposition}
 \label{prop:gw}
Assume that $T$ is a critical Galton--Watson tree whose offspring distribution is in the domain of attraction of an $\alphaGW$-stable law, $\alphaGW\in(1,2]$, and let $D = \alphaGW/(\alphaGW-1)$. 
Then the following statements hold:
 \begin{enumerate}
  \item For $n\in\N$, let $T_n$ be a random tree following the law of $T$ conditioned on having $n$ vertices. Then Assumptions (Tn1) and (Tn2) hold with the same $D$.
  \item Assumptions (T1) and (T2) above hold with the same $D$, $\beta = 1/\alphaGW$ and every $\gammabar \in \R$. The asymptotic on $1-g_V(1-s)$ as $s\to0$ is given in \eqref{eq:gV_GW} and \eqref{eq:gV_GW_finite_variance}.
 \end{enumerate}
\end{proposition}

We remark that the ``dimension'' $D = \alphaGW/(\alphaGW-1)$ in Proposition~\ref{prop:gw} above indeed coincides with the Hausdorff (or packing) dimension of the $\alpha_T$-stable tree \cite{Duquesne2004}, which is the scaling limit of the trees $T_n$ from Proposition~\ref{prop:gw} \cite{Duquesne2003}.

Apart from critical Galton--Watson trees, which cover the range $D\in[2,\infty)$, we remark that the results in this paper also allow to treat supercritical Galton--Watson trees, which correspond to $D=\infty$.

\paragraph{Overview of the remainder of the article.} Section~\ref{sec:discussion} presents the heuristics that underly the proofs of Theorems~\ref{th:1} and \ref{th:2}. It also discusses possible generalizations. Section~\ref{sec:tree-bound} is devoted to the proof of Proposition~\ref{prop:tree_bound}. Theorems~\ref{th:1} and \ref{th:2} are proven in Section~\ref{sec:tree-indexed-reg-var}. The case of Galton--Watson trees is considered in Section~\ref{sec:gw} where Proposition~\ref{prop:gw} is proven. Finally, an appendix, Section~\ref{sec:reg-var}, recalls standard properties of regularly varying functions.

\section{Heuristics and discussion}
\label{sec:discussion}

In this article, we consider tree-indexed random walks with heavy tails in a regime where the total maximal displacement is attained by atypically large fluctuations on certain branches.
At the heart of the results are therefore large deviation results for heavy-tailed random walks. Recall that a random variable $X$ is said to follow a \emph{subexponential distribution} if
\[
 \lim_{x\to\infty} \frac{\P(S_n > x)}{n\P(X>x)} = 1\quad\text{for all $n\in\N$,}
\]
where $S_n = X_1+\cdots +X_n$ with $X_1,X_2,\ldots$ iid copies of $X$. A huge body of literature is devoted to the problem of finding (optimal or nearly optimal) sequences $x_n\to\infty$ such that $\P(S_n > x) \approx n\P(X>x)$ for all $x\ge x_n$.
The literature on this topic is quite overwhelming\footnote{Classical references are \cite{Linnik1961,Linnik1961a,Heyde1967,Nagaev1969,Nagaev1969a}, for more see \cite[Section 8.6]{EmbKluMik}, \cite{Mikosch1998} and \cite{Denisov2008}. Treatments of general distributions with regularly varying tails appear in \cite{Durrett1979,Cline1998}. Denisov, Dieker and Shneer \cite{Denisov2008} give a uniform and fairly insightful treatment of general subexponential distributions. }, due in part to the large variety of subexponential distributions leading to a substantial number of treatments differing in results and/or techniques. 

For our Proposition~\ref{prop:tree_bound}, we use an ``exponential bound'' from \cite{Denisov2008} on suitably truncated random walks, which we recall in Equation~\eqref{eq:DDS_bound} below. This is the only random walk bound we use in this paper. We combine it with the simple observation that the maximal displacement of a tree-indexed random walk is approximately equal to the size of the maximal jump if the following two events happen:
\begin{enumerate}
 \item No two large jumps occur on the same branch of the tree with high probability.
 \item The contribution of the small jumps are asymptotically negligible.
\end{enumerate}
Quantifying this leads to Proposition~\ref{prop:tree_bound}. Note that the precise bound given in the statement of  Proposition~\ref{prop:tree_bound} will turn out to be not so important, but rather the assumption that $y\ge b_H$ which will need to be verified for the values of $y$ we will be interested in. This assumption exactly corresponds to requiring that the typical value of the maximal displacement or of the biggest jump is is typically much larger than the values of a random walk along a (fixed) branch of the tree.

For Theorems~\ref{th:1} and \ref{th:2} we make explicit such a regime in the case of regularly varying displacement. Under Assumption (XR), this means that we require that $V^{1/\alpha+o(1)} \gg b_H$, where $b_n$ is a natural scale sequence for the random walk. This amounts to the assumption $D > \Dcrit(X)$ in Theorems~\ref{th:1} and \ref{th:2}. 


We do not go into further details of the proofs of Theorems~\ref{th:1} and \ref{th:2} here. However, we elaborate briefly on the possible improvements of Proposition~\ref{prop:tree_bound}. First, the exponential bound \eqref{eq:DDS_bound} might be replaced by a better bound. Indeed, although it is fairly efficient for values of $x$ such that $\P(S_n > x)$ is not too small, it is quite bad for those $x$ for which $\P(S_n > x)$ is very small\footnote{For example, in case of regularly varying tails it is known that the sum over the jumps of size at most $(1-\ep)x$, $\ep>0$, is negligible for large enough $x$ compared to the non-truncated sum \cite{Durrett1979,Cline1998}, a fact which is not apparent from \eqref{eq:DDS_bound}}. Although this is not too much of a problem in the case of regularly varying tails, it is in fact disastrous in the case of stretched exponential tails $\P(X>x) \approx \exp(-x^r)$ when  $r$ is small (see \cite[Theorem~3]{Gantert2000} for bounds on truncated random walks with streched exponential tails). The reason why we chose \eqref{eq:DDS_bound} is its simplicity (both statement and proof) and its generality: we are not aware of any other work which allows to treat such a large class of distributions. 

As a second improvement of Proposition~\ref{prop:tree_bound}, one might relax the first of the two points mentioned above. Namely, instead of just throwing away the event where two large jumps occur on the same branch of the tree, one might instead separate the jumps into large jumps and small jumps, ignore the small jumps and then consider the skeleton of the tree consisting of the large jumps. Then, by a certain induction, one might use bounds on this smaller tree to get bounds on the original, larger tree. Working out such an argument would give better quantitative bounds on the difference between the maximal displacement and the maximal jump size. For example it would give good bounds on the range of values $x$ (in terms of $n$) such that $\P(\Smax_n> x)/\P(\Xmax_n>x) \in[1-\ep,1+\ep]$ in Theorem~\ref{th:1}, for $\ep>0$. However, in order to keep the current proof as simple as possible, we did not pursue this argument and leave it open for future work.

%
%

\section{Proof of Proposition~\ref{prop:tree_bound}}
\label{sec:tree-bound}

We start with a large deviation estimate for random walks. Let $X_1,X_2,\ldots$ be iid copies of the random variable $X$ and define $S_n = X_1+\cdots+X_n$ and for $y\ge0$, 
\[
 S_n^{(y)} = \sum_{k=1}^n X_k\Ind_{(|X_k|\le y)}.
\]
We then have the following lemma:

\begin{lemma}
\label{lem:DDS_bound}
Let $(b_n)_{n\in\N}$ be a natural scale sequence. Then there exists $C\in(0,\infty)$ (depending only on the law of $X$ and on the sequence $(b_n)_{n\in\N}$), such that
\[
\forall n\in\N\,\forall x\ge0\,\forall y\ge b_n: \P(|S^{(y)}_n|>x) \le C \exp(-x/y).
\]
\end{lemma}
\begin{proof}
The lemma is a simple extension of the following bound \cite[Lemma~2.1]{Denisov2008}:
\begin{equation}
 \label{eq:DDS_bound}
 \exists C\in(0,\infty)\,\forall n\in\N\,\forall x\ge0\,\forall y\ge b_n: \P(|S_n|>x,|X_1|\le y,\ldots,|X_n|\le y) \le C \exp(-x/y).
\end{equation}
In order to use \eqref{eq:DDS_bound}, we decompose:
 \begin{align*}
  \P(|S^{(y)}_n|>x) &= \sum_{I\subset \{1,\ldots,n\}} \P(|S^{(y)}_n|>x,\,\forall i\in I:|X_i|\le y,\,\forall j\not\in I:|X_j| > y)\\
  &= \sum_{k=0}^n \binom n k \P(|S_{n-k}| > x,\,|X_1|\le y,\ldots,|X_{n-k}|\le y) \P(|X|>y)^k.
 \end{align*}
Since $b_n$ is increasing by assumption, we have $y\ge b_{n-k}$ for every $k\in\{0,\ldots,n\}$. With \eqref{eq:DDS_bound}, this gives for all $n\in\N,\, x\ge0,\, y\ge b_n$:
\begin{align*}
 \P(|S^{(y)}_n|>x) &\le C \exp(-x/y) \sum_{k=0}^n \binom n k \P(|X|>y)^k\\
 &\le C \exp(-x/y) (1+\P(|X| > b_n))^n \\
 &\le C\exp(-x/y) \exp(n\P(|X|>b_n)).
\end{align*}
From the proof of the lemma in \cite[Section~IX.7]{Feller1971}, we have $\sup_n n\P(|X|>b_n) < \infty$. This yields the lemma.
\end{proof}

\begin{proof}[Proof of Proposition~\ref{prop:tree_bound}]
Fix $z\ge b_H$ and $y\ge 0$. Define the events
 \begin{align*}
  G_1 &= \{\forall u,v\in\mathcal V^*: \tif |X_u| > z\text{ and }|X_v|> z\text{, then } u \not< v\},\\
  G_2 &= \{\forall v\in\mathcal V: |S_v^{(z)}| \le y\}.
 \end{align*}
Here, similarly to the definition at the beginning of the section, we define $S_v^{(z)} = \sum_{\rho \ne u \le v} X_u\Ind_{(|X_u|\le z)}.$
It is clear that
\[
 \text{ on }G_1\cap G_2,\quad \DeltaSX \le y.
\]
It therefore suffices to bound $\P(G_1^c)$ and $\P(G_2^c)$. By a union bound, we have
\[
 \P(G_1^c) \le \sum_{u,v\in\mathcal V^*,\,u < v} \P(|X| > z)^2 \le \frac{HV}{2} \P(|X| > z)^2.
\]
Furthermore, again by a union bound,
\[
 \P(G_2^c) \le \sum_{v\in\mathcal V^*} \P(|S_v^{(z)}| > y) \le CV \exp(-y/z),
\]
where the last inequality follows from Lemma~\ref{lem:DDS_bound} and the fact that $z \ge b_n$ for all $n\le H$ since $b_n$ is non-decreasing in $n$. The two previous inequalities then yield
\begin{equation*}
 \P((G_1\cap G_2)^c) \le \P(G_1^c) + \P(G_2^c) \le \frac {HV} 2 \P(|X| > z)^2 + CV\exp(-y/z).
\end{equation*}
This finishes the proof.
\end{proof}

\section{Regularly varying displacement: proofs of Theorems~\ref{th:1} and~\ref{th:2}}
\label{sec:tree-indexed-reg-var}

Throughout this section, we will make use of some classic results on regularly varying functions (in particular, Potter's bounds), readers not familiar with this theory may refer to Section~\ref{sec:reg-var} where these results are recalled.

Assume from now on that Assumption (XR) holds. Standard results on triangular arrays \cite[Section~XI.7]{Feller1971} or the domain of attraction of stable laws \cite[Section~XVII.5]{Feller1971}, together with Potter's bounds easily give that for every $\ep>0$ the following sequence is a natural scale sequence for the random walk $(S_n)_{n\in\N}$:
\begin{equation}
\label{eq:bn}
 b_n^\ep = 
 \begin{cases}
  n^{1/(2\wedge \alpha) + \ep},& \text{if $\E[|X|]<\infty$ and $\E[X]=0$}\\
  n^{1/(1\wedge \alpha) + \ep},& \text{otherwise.}
 \end{cases}
\end{equation}

We now first simplify Lemma~\ref{lem:DDS_bound} to the current setting:
\begin{corollary}
\label{cor:tree_bound_reg_var}
 For every $\ep>0$, there exists $C=C(\ep)\in(0,\infty)$, such that for every tree $T$, for every $y \ge b_H^\ep$,
\[
  \P(\DeltaSX > y)
  \le C HV y^{-(2-\ep)\alpha}.
\]
\end{corollary}
\begin{proof}
 Fix $\ep\in(0,1)$. Let $\delta\in(0,\ep/2)$. By the assumption on $\P(|X| > x)$ and Potter's bounds, there exists $C'=C'(\ep,\delta)$, such that for all $y> 0$,
 \[
  \P(|X| > y^{1-\delta})^2+\exp(-y^{\delta})\le C' y^{-(2-\ep)\alpha}.
 \]
 Now choose $\delta = \delta(\ep,\alpha) \in (0,\ep/2)$ such that $(b_n^\ep)^{1-\delta}$ is a natural scale sequence. By Proposition~\ref{prop:tree_bound} applied with $z = y^{1-\delta}$, we then have for every tree $T$ and for every $y\ge b_H^\ep$, 
\[
 \P(\DeltaSX > y) \le \frac {HV} 2 \P(|X| > y^{1-\delta})^2 + CV\exp(-y^\delta) \le C'(C+H)V y^{-(2-\ep)\alpha},
\]
where $C$ is the constant from Proposition~\ref{prop:tree_bound}. This proves the corollary for all trees of height $H\ge1$. For $H=0$ the bound trivially holds. This finishes the proof.
\end{proof}

We are now ready for the proofs of Theorems~\ref{th:1} and \ref{th:2}.


\begin{proof}[Proof of Theorem~\ref{th:1}]
 We trivially have $\Xmax_n \le \Xabsmax_n$. In order to prove the first statement of the theorem, it is therefore enough to show that 
\begin{equation}
\label{eq:conv_probab}
\lim_{n\to\infty} \frac{\DeltaSX_n}{\Xmax_n} = 0,\quad\text{in probability.}
\end{equation}

We first give a lower bound on $\Xmax_n$. By independence, we have for all $x$,
\begin{equation}
\label{eq:Xmax_cdf}
 \P(\Xmax_n \le x\,|\,T_n) = \P(X\le x)^{V_n}.
\end{equation}
Potter's bounds then give for every $\ep>0$ and for large $x$,
\begin{equation*}
\P(\Xmax_n \le x\,|\,T_n) = (1-x^{-(\alpha+\ep)})^{V_n} \le \exp(-V_n x^{-(\alpha+\ep)}).
\end{equation*}
Applying this with $x = V_n^{1/(\alpha+2\ep)}$ gives for every $\ep>0$,
\begin{equation}
\label{eq:poste}
 \P(\Xmax_n \le V_n^{1/(\alpha+\ep)}) \to 0,\quad\tas n\to\infty.
\end{equation}

Now let $b_n^\ep$ be as in \eqref{eq:bn}, so that $b_n^\ep = n^{\eta + \ep}$, where
\begin{equation}
 \label{eq:eta}
 \eta = \begin{cases}
                 1/(2\wedge\alpha) & \text{if $\E[|X|]<\infty$ and $\E[X]=0$}\\
                 1/(1\wedge \alpha) & \text{otherwise.}
                \end{cases}.
\end{equation}
Note that by the assumption on $\alpha$ and the assumption $D>1$, we have \(\eta/D < 1/\alpha\). In particular, for every $\ep>0$ small enough, we have for every $h,v\in\N$
\begin{equation}
\label{eq:colis}
 h \le v^{1/D+\ep} \Rightarrow b^\ep_{h} \le v^{1/(\alpha+\ep)}\quad\text{for $v$ large enough.}
\end{equation}

By Corollary~\ref{cor:tree_bound_reg_var} (applied to the tree-indexed random walk conditioned on the tree $T_n$), we now have for every $\ep > 0$, for some $C=C(\ep)<\infty$,
\begin{align}
\nonumber
 \P(\DeltaSX_n > V_n^{1/(\alpha+\ep)}\,|\,T_n)\Ind_{(H_n \le V_n^{1/D+\ep},\,b^\ep_{H_n} \le V_n^{1/(\alpha+\ep)})} &\le CH_nV_n V_n^{-(2-\ep)\alpha/(\alpha+\ep)}\Ind_{(H_n \le V_n^{1/D+\ep})}\\
 \label{eq:poste2}
 &\le C V_n^{1/D + 1 - 2 +C' \ep},
\end{align}
for some constant $C' = C'(\alpha)$. Since $D > 1$ by (Tn2) and $V_n\to+\infty$ in probability by (Tn1), the previous equations \eqref{eq:poste2} and \eqref{eq:colis} show that for every $\ep>0$ small enough,
\begin{equation}
\label{eq:final}
 \P(\DeltaSX_n > V_n^{1/(\alpha+\ep)},\,H_n \le V_n^{1/D+\ep}) \to 0,\quad\tas n\to\infty.
\end{equation}
Summing up the above equations, we have for some $\ep>0$ small enough,
\begin{multline*}
 \P(\DeltaSX_n > (\Xmax_n)^{(\alpha+\ep/2)/(\alpha+\ep)}) \le \\
 \P(\DeltaSX_n > V_n^{1/(\alpha+\ep)},\,H_n \le V_n^{1/D+\ep}) + \P(\Xmax_n \le V_n^{1/(\alpha+\ep/2)}) + \P(H_n > V_n^{1/D+\ep}),
\end{multline*}
and all three terms go to zero as $n\to\infty$ by \eqref{eq:final}, \eqref{eq:poste} and assumption (Tn2), respectively. Since $\Xmax_n\to\infty$ in probability by assumption (Tn1), this proves \eqref{eq:conv_probab} and therefore finishes the proof of the first statement of the theorem. The second statement follows easily from \eqref{eq:Xmax_cdf} and standard arguments.
\end{proof}

\begin{proof}[Proof of Theorem~\ref{th:2}]
 We only prove the statements involving $\Xmax$. The first statement involving $\Xabsmax$ then immediately follows since $\Xmax\le \Xabsmax$ by definition, and the second statement involving $\Xabsmax$ is proven similarly.
 
 As in \eqref{eq:Xmax_cdf} we have by independence,
 \begin{equation}
  \label{eq:MX_tail}
  \P(\Xmax > x) 
  = 1-\E[\P(X \le x)^V] = 1-g_V(1-\P(X>x)).
 \end{equation}
With Assumptions (XR) and (T1), this implies that $\P(\Xmax > x)$ is regularly varying at $\infty$ (with index $-\beta\alpha$), such that in particular, for every $\delta(x) \to 0$ as $x\to\infty$, by the uniform convergence theorem for regularly varying functions \cite[Theorem 1.2.1]{Bingham1987},
\begin{equation}
 \label{eq:MX_tail_reg}
 \P(\Xmax > (1-\delta(x))x) \sim \P(\Xmax > x) \sim \P(\Xmax > (1+\delta(x))x),\quad\tas x\to\infty.
\end{equation}

Fix\footnote{Note that $\gamma$ exists, since by Assumption (T2), $(1\wedge \alpha)\gammabar > (1\wedge\alpha)(1\vee\alpha) = \alpha$.} $\gamma \in (\alpha/\gammabar,1\wedge \alpha)$. For $\ep > 0$ and $x>0$, set 
\[
G_T := \{V\le x^{\gamma}\}\cup\{H \le V^{\frac 1 D + \ep},\,V \le x^{\alpha+\ep}\}
\]
(note that $G_T$ depends on $\ep$ and $x$ but that we suppress this from the notation for readability).
Fix a positive function $\delta(x)$ converging to $0$ slower than polynomially, e.g. $\delta(x) = 1/\log(2+x)$. 
We claim that for small enough $\ep$,
\begin{align}
 \label{eq:GT}
 &\P(G_T^c) = o(\P(\Xmax > x)) && \tas x\to\infty,\\
\label{eq:finally}
 &\P(\DeltaSX > \delta(x)x,\,G_T) = o(\P(\Xmax > x))  && \tas x\to\infty.
\end{align}
Let us show how Equations \eqref{eq:MX_tail}, \eqref{eq:MX_tail_reg}, \eqref{eq:GT} and \eqref{eq:finally} together imply the theorem. First, by \eqref{eq:GT} and \eqref{eq:finally}, we have
\begin{equation}
\P(\DeltaSX > \delta(x)x) \le \P(\DeltaSX > \delta(x)x,\,G_T) + \P(G_T^c) = o(\P(\Xmax > x)),
\label{eq:addition}
\end{equation}
which implies in particular the last statement of the theorem. As for the other statements, note that we have
\[
\ton \{\DeltaSX \le \delta(x)x\},\quad \{\Xmax > (1+\delta(x))x\} \subset \{\Smax > x\} \subset \{\Xmax > (1-\delta(x))x\}.
\]
Hence, by \eqref{eq:addition}, \eqref{eq:MX_tail_reg} and \eqref{eq:MX_tail},
\[
\P(\Smax > x) \sim \P(\Xmax > x) = 1-g_V(1-\P(X>x)).
\]
The remaining statements of the theorem follow similarly.

It remains to prove \eqref{eq:GT} and \eqref{eq:finally}. Let us start with \eqref{eq:GT}. We have
\begin{align*}
 \P(G_T^c) \le \P(V > x^{\alpha+\ep}) + \P(H > V^{\frac 1 D + \ep},\,V > x^{\gamma})
\end{align*}
by Assumption (T1) and Karamata's Tauberian theorem\footnote{In the case $\beta = 1$, use that $\E[\min(X,x)] = \int_0^x \P(X > y)\,dy$ and that $\P(X>y)$ is decreasing in $y$.} (Fact~\ref{fact:karamata}), $\P(V\ge y)$ is dominated by a regularly varying function with index $-\beta$ at $+\infty$. Potter's bounds then give that the first summand is bounded by $x^{-\beta\alpha -\beta\ep/2}$ for large $x$. As for the second summand, by Assumption (T2) it is bounded by $x^{-\beta\gamma\gammabar}$ for large $x$, and $\gamma\gammabar > \alpha$ by definition of $\gamma$. Since $\P(\Xmax > x)$ is regularly varying with index $-\beta\alpha$, this readily implies \eqref{eq:GT}.

It remains to show \eqref{eq:finally}. We want to apply Corollary~\ref{cor:tree_bound_reg_var} and thus need to show:
\begin{equation}
\label{eq:xgebH}
 \forall \ep,\ep'>0\text{ small enough, }\forall x\ge 1, \text{ we have }x\ge b_H^{\ep'}\text{ on the event $G_T$,}
\end{equation}
where $b_H^{\ep'}$ is defined in \eqref{eq:bn}. Let $\eta$ be as in \eqref{eq:eta}, such that $b_H^{\ep'} = H^{\eta+\ep'}$.
Then, on the event $\{V\le x^{\gamma}\}$, since $H\le V$, we have $b_H^{\ep'} \le x^{\gamma(\eta+\ep')}$. Now note that $\gamma\eta < (1\wedge\alpha)/(1\wedge\alpha) = 1$. Hence, for $\ep'$ small enough and $x\ge 1$, we have $x\ge b_H^{\ep'}$ on the event $\{V\le x^{\gamma}\}$.

Now suppose that $H \le V^{\frac 1 D + \ep}$ and $V \le x^{\alpha+\ep}$, such that $H \le x^{\alpha/D + C\ep}$, with $C=1/D+\alpha+\ep$. By the assumption on $\alpha$ and the assumption $D>1$, we have \((\alpha/D)\eta < 1\). In particular, for every $\ep',\ep>0$ small enough, we have $b_H^{\ep'} \le x^{(\alpha/D + C\ep)(\eta+\ep')}\le x$ for $x\ge 1$. This proves \eqref{eq:xgebH}.

Now suppose for the rest of the proof that $\ep,\ep'>0$ are such that both 
\eqref{eq:GT}
and \eqref{eq:xgebH} hold. 
Then $\delta(x)x \ge b_H^{\ep''}$ on the event $G_T$ for large $x$, for every $\ep''<\ep'$. Corollary~\ref{cor:tree_bound_reg_var} then yields for every $\ep''<\ep'$, for large $x$, on the event $G_T$, with $C=C(\ep')$,
\[
 \P(\DeltaSX > \delta(x)x\,|\,T) \le CHV(\delta(x)x)^{-(2-\ep'')\alpha} \le CHV x^{-(2-2\ep'')\alpha}.
\]
Integrating over $T$ and using the fact that $H\le V$, this gives for every $\ep''<\ep'/2$, for large $x$,
\begin{equation}
\label{eq:hair}
 \P(\DeltaSX > \delta(x)x,\,G_T) \le C\E\left[V^2\Ind_{(V\le x^{\gamma})} + V^{\frac 1 D + 1 + \ep} \Ind_{(V\le x^{\alpha+\ep})}\right]x^{-(2-\ep'')\alpha}.
\end{equation}
Since $\P(V > y)$ is dominated by a regularly varying function with index $-\beta$ at $+\infty$ (see above), this gives for every $\widetilde \ep>0$ and every $y>0$ and $r>\beta$,
\[
 \E[V^r\Ind_{(V\le y)}] \le r \int_0^y z^{r-1}\P(V>z)\,dz \le Cr\int_0^y z^{r-1-\beta+\widetilde\ep}\,dz\le C\frac r {r-\beta} y^{r-\beta+\widetilde\ep}.
\]
Together with \eqref{eq:hair}, this gives for all $\ep>0$ small enough, all $\ep'>0$ and all $x$ large enough,
\[
  \P(\DeltaSX > \delta(x)x,\,G_T) \le C(x^{\gamma(2-\beta)-2\alpha} + x^{\alpha(\frac 1 D + 1 - \beta - 2)+C'\ep})x^{\ep'},
\]
with $C' = 1/D+1+\alpha+\ep$.
Now, since $\gamma < 1\wedge\alpha\le \alpha$, we have $\gamma(2-\beta)-2\alpha < -\beta\alpha$. Furthermore, since $D > 1$, we have $\alpha(\frac 1 D + 1 - \beta - 2) < -\beta\alpha$ as well. Choosing $\ep$ and $\ep'$ small enough in the previous inequality and using again the fact that $\P(\Xmax > x)$ is regularly varying with index $-\beta\alpha$ yields \eqref{eq:finally} and thus finishes the proof of the theorem.
%
%
%
\end{proof}

\section{Galton--Watson trees: proof of Proposition~\ref{prop:gw}}
\label{sec:gw}

Throughout the section, we denote by $T$ a critical Galton--Watson tree with offspring distribution $\mathbf p = (p_n)_{n\in\N}$ in the domain of attraction of an $\alphaGW$-stable law, $\alphaGW\in(1,2]$. We also denote by $H$ its height and by $V$ the number of its vertices\footnote{We include the root here because it makes the formulae below simpler. It is clear that this will not affect the validity of Assumptions (T1) and (T2).}. Since $\mathbf p = (p_n)_{n\in\N}$ is in the domain of attraction of an $\alphaGW$-stable law, there exists a sequence $a_n = L(n)n^{1/\alphaGW}$ with a slowly varying function $L$ such that $a_n^{-1} (Z_1+\cdots+Z_n - n)$ converges in distribution (to an $\alphaGW$-stable law), where $Z_1,Z_2,\ldots$ are iid copies of a random variable $Z$ with law $\mathbf p$. Note that this means that there exists\footnote{We could of course assume  w.l.o.g. $\lambda = 1$, but we keep the general form for convenience.} $\lambda > 0$, such that \cite[Theorem~2.6.1]{Zolotarev1986}
\begin{equation}
 \label{eq:laplace}
 \forall t\ge 0:\lim_{n\to\infty} \E[e^{-t(Z-1)/a_n}]^n = e^{\lambda t^{\alphaGW}}.
\end{equation}

For the second part of Proposition~\ref{prop:gw}, we will need asymptotics for the generating function $g_V$ of $V$. These are well-known, but we establish them here for completeness. Recall the following formula\footnote{The same formula is satisfied by the generating function of the hitting time $\tau$ of $-1$ of the left-continuous random walk $S_n = Z_1+\cdots+Z_n - n$, see e.g.~\cite[p234]{Spitzer1976}. In fact, it is well-known that $V$ and $\tau$ are equal in law; this follows from an encoding of the Galton--Watson tree through its \emph{{\L}ukasiewicz path}, see e.g.~\cite{LeGall2012}.} due to Good~\cite{Good1949}:
\begin{equation}
 \label{eq:good}
 g_V(s) = sg_Z(g_V(s)),\quad\text{for $|s| \le 1$},
\end{equation}
where $g_Z$ is the generating function of the random variable $Z$. It will be more useful to translate this formula in terms of log-Laplace transforms. For $t\ge 0$, let $\kappa_V(t) = \log \E[e^{-t V}]$ and $\kappa_{Z-1}(t) = \log \E[e^{-t(Z-1)}]$. Then \eqref{eq:good} becomes
\begin{equation}
 \label{eq:good2}
 \kappa_V = -\kappa_{Z-1}^{-1}\quad\text{on $[0,\infty)$,}
\end{equation}
where the existence of the inverse follows from simple convexity arguments.

Equation \eqref{eq:laplace} now gives
\[
 \forall t\ge 0:\quad \kappa_{Z-1}(t/a_n) \sim \lambda t^{\alphaGW}/n,\quad\tas n\to\infty,
\]
from which one easily sees that $\kappa_{Z-1}$ is regularly varying at zero\footnote{This can also be obtained from the characterization of the domain of attraction of stable laws in terms of truncated second moments and a tail balance condition \cite[Section XVII.5, Theorem 2]{Feller1971}, together with a Tauberian theorem \cite[Theorem 8.1.6]{Bingham1987}.} with index $\alphaGW$. Furthermore, setting $t=\lambda^{-1/\alphaGW}$ in the above equation and taking inverses on both sides of the equation gives
\begin{equation}
 \label{eq:kappa_inv}
 \kappa_{Z-1}^{-1}(t) \sim \frac{t^{1/\alphaGW}}{\lambda^{1/\alphaGW}L(1/t)},\quad\tas t\to\infty.
\end{equation}
Equations~\eqref{eq:good2} and \eqref{eq:kappa_inv} now readily yield\footnote{This formula can also be obtained by the so-called Kemperman's formula for the hitting time of a left-continuous random walk on $\Z$ (see \cite[p234]{Spitzer1976} or \cite{LeGall2012}) together with local limit theorems for random walks \cite[Theorem~4.2.1]{Ibragimov1971}, an explicit expression of the density at 0 of the $\alphaGW$-stable law (see e.g.~\cite[Section~2.2, Corollary~2]{Zolotarev1986}) and a Tauberian theorem. This way, one obtains several multiplicative factors involving the Gamma function, the sine function and the constant $\pi$ which mysteriously cancel by Euler's reflection formula for the Gamma function. In our opinion, the approach presented here is more transparent.}
\begin{equation}
 \label{eq:gV_GW}
1-g_V(1-s) \sim \frac{s^{1/\alphaGW}}{\lambda^{1/\alphaGW}L(1/s)},\quad\tas s\to 0.
\end{equation}
In particular, if $\sigma^2 = \operatorname{Var}(Z_1) < \infty$, then setting $\alphaGW = 2$, $L \equiv \sigma$ and $\lambda = 1/2$, we have
\begin{equation}
 \label{eq:gV_GW_finite_variance}
1-g_V(1-s) \sim \frac{\sqrt 2}{\sigma} \sqrt{s},\quad\tas s\to 0.
\end{equation}
%
%

We can now turn to the proof of Proposition~\ref{prop:gw}.

\begin{proof}[Proof of Proposition~\ref{prop:gw}]
The first part follows directly from existing results in the literature: It is known\footnote{See \cite[Theorem~3.1]{Duquesne2003}. The assumption of aperiodicity in that paper is not needed.} that conditioned on $V=n$, the random variable $(a_n/ n) H$ converges in law to a non-degenerate random variable, as $n\to\infty$.
Since $n/a_n = n^{1-1/\alphaGW+o(1)} = n^{1/D+o(1)}$, this immediately proves the first part of the proposition.

We now turn to the second part of the proposition. Assumption (T1) follows directly from \eqref{eq:gV_GW}. It remains to check that (T2) holds for every $\gammabar\in\R$. This is in fact a direct consequence of the results in \cite{Kortchemski2015} on the height of Galton--Watson trees conditioned on its number of vertices. However, since we only need weaker results, we present here for completeness a simple and transparent proof. We recall the following construction due to Geiger \cite{Geiger1999} of the Galton--Watson tree conditioned on the event that its height is at least $k\in\N$. Define $c_i = \P(H \ge i-1)/\P(H\ge i)$. Let $(A_1,B_1),\ldots,(A_{k},B_{k})$ be independent pairs of integer r.v. with law
\[
 \P(A_i = a,\,B_i = b) = c_ip_b\P(H < i-1)^{a-1}\Ind_{(1\le a\le b)}.
\]
Then the tree $T$ conditioned on $H\ge k$ can be constructed as follows:
\begin{itemize}
 \item There exists a ray of length $k$, called the \emph{spine}, starting from the root.
 \item The vertex on the spine at generation/depth/height $i$, $i=0,\ldots,k-1$, has $B_{k-i}-1$ children off the spine, out of which, independently,
 \begin{itemize}
  \item $A_{k-i} - 1$ children spawn copies of $T$ conditioned on $H < k-(i+1)$,
  \item $B_{k-i} - A_{k-i}$ children spawn copies of $T$.
 \end{itemize}
 \item The vertex on the spine at generation $k$ spawns a copy of $T$.
\end{itemize}

Ignoring in the above construction the $A_{k-i}-1$ copies of $T$ conditioned on $H < k-(i+1)$ in each generation, we obtain that the number of vertices in the above tree is stochastically bounded from below by the sum of $\sum_{i=1}^k B_i-A_i$ copies of $V$. Hence, if $V^{(1)},V^{(2)},\ldots$ denote iid copies of $V$, then
\begin{equation}
 \label{eq:one}
 \P(V \le n\,|\,H\ge k) \le \P\left(\max_{j=1,\ldots,\sum_{i=1}^k B_i-A_i} V^{(j)} \le n\right) = g_{\sum_{i=1}^k B_i-A_i}(\P(V \le n)),
\end{equation}
where for a random variable $Y$ we denote by $g_Y$ its generating function.
Now let $(A,B)$ be a pair of integer random variables such that
\[
 \P(A = a,B = b) = p_b\Ind_{(1\le a\le b)},
\]
so that $B$ is distributed according to the size-biased distribution of $\mathbf p$ and conditionally on~$B$, $A$ is uniform in $\{1,\ldots,B\}$ (this law also appears in \cite{Geiger1999}). For every $s\in[0,1]$, we then have by the definition of $(A_i,B_i)_{i=1,\ldots,k}$,
\begin{align}
 \nonumber
 g_{\sum_{i=1}^k B_i-A_i}(s) &= \prod_{i=1}^k \sum_{j=0}^\infty \P(B_i-A_i = j) s^j\\
 \nonumber
 &\le \left(\prod_{i=1}^k c_i\right) \left(\sum_{j=0}^\infty \P(B-A = j) s^j\right)^k\\
 \label{eq:two}
 &= \frac 1 {\P(H\ge k)} g_{B-A}(s)^k.
\end{align}
We have,
\begin{align*}
 g_{B-A}(s) &= \sum_{j=0}^\infty \P(B-A = j) s^j = \sum_{j=0}^\infty \left(\sum_{l=j+1}^\infty p_l\right) s^j = (1-s)^{-1}\sum_{l=0}^\infty p_l (1-s^l) \\
 &= (1-s)^{-1}(1-g_Z(s)).
\end{align*}
Now, since $\kappa_{Z-1}(t)$ is regularly varying at zero with index $\alphaGW$ (see above), we have
\[
 g_Z(s) = s + (1-s)^{\alphaGW+o(1)} = 1 - (1-s) + (1-s)^{\alphaGW+o(1)},\quad \tas s\to 1.
\]
Hence,
\begin{equation}
\label{eq:three}
 g_{B-A}(s) = 1- (1-s)^{\alpha_T-1+o(1)},\quad \tas s\to 1.
\end{equation}
Equations~\eqref{eq:one}, \eqref{eq:two} and \eqref{eq:three} now give for large $n$,
\begin{align*}
 \P(V\le n,\,H \ge k) &\le g_{B-A}(\P(V\le n))^k = (1-\P(V>n)^{\alpha_T-1+o(1)})^k\\
 &\le \exp(-k\P(V>n)^{\alpha_T-1+o(1)}),
\end{align*}
such that by \eqref{eq:gV_GW} and Fact~\ref{fact:karamata},
\begin{equation}
 \label{eq:VHnk}
 \P(V\le n,\,H \ge k) \le \exp(-k n^{-1/D+o(1)}),\quad\tas n\to\infty.
\end{equation}

Now fix $\ep>0$ and $\rho > 1$. We have with $k = k(n) = n^{1/(1/D+\ep)}$,
\begin{align*}
 \P(H > V^{\frac 1 D + \ep},\,V \ge n) &\le  \P(H > V^{\frac 1 D + \ep},\,H \ge k)\\
 & = \sum_{i=0}^\infty \P(H > V^{\frac 1 D + \ep},\,k^{\rho^i}\le H < k^{\rho^{i+1}})\\
 &\le \sum_{i=0}^\infty \P(V < k^{\rho^{i+1}/(1/D+\ep)},\,k^{\rho^i}\le H < k^{\rho^{i+1}})\\
 &\le \sum_{i=0}^\infty \P(V < (k^{\rho^i})^{\rho/(1/D+\ep)},\,k^{\rho^i}\le H)\\
 &\le \sum_{i=0}^\infty \exp(-k^{\rho^i} k^{-\rho^i(\rho/(1+D\ep)+o(1))}) &&\text{(by \eqref{eq:VHnk}).}
\end{align*}
Choosing $\rho < 1+D\ep$ yields for some $\eta > 0$,
\begin{equation*}
\P(H > V^{\frac 1 D + \ep},\,V \ge n) \le  \exp(-n^{\eta}),\quad\text{for large $n$.}
\end{equation*}
This proves (T2) for every $\gammabar \in\R$.
\end{proof}

\appendix
\section{Regularly varying functions}
\label{sec:reg-var}

We recall here known facts about regularly varying functions.

\begin{definition*}
A function $f:\R_+\to\R$ is said to \emph{vary regularly at $\infty$ (at zero) with index $\lambda\in\R$} if for every $y>0$,
\[
 \frac{f(xy)}{f(x)} \to y^\lambda,\quad\tas x\to \infty\ (x\to0).
\]
If $\lambda=0$, the function is also called \emph{slowly varying}. The definition is extended to sequences by linear interpolation (say).
\end{definition*}
It follows readily from the definition that the composition of regularly varying functions is again a regularly varying function (whose index is the product of the indices).

The following result is often used in this article, sometimes without mentioning it explicitly:
\begin{fact}[Potter's bounds {\cite[Theorem~1.5.6]{Bingham1987}}]
\label{fact:potter}
 If $f$ is regularly varying of index $\lambda$ at $\infty$, then for every $C>1$ and $\delta > 0$ there exists $x_0 = x_0(C,\delta)$, such that 
\[
\frac{f(y)}{f(x)} \le C\max((y/x)^{\lambda+\delta},(y/x)^{\lambda-\delta}),\quad \text{for all $x\ge x_0$, $y\ge x_0.$}
\]
\end{fact}

Regularly varying functions play an important role in Tauberian theorems, of which we will use the following form:
\begin{fact}[Karamata's Tauberian theorem, extended form {\cite[Corollary~8.1.7]{Bingham1987}}]
\label{fact:karamata}
 Let $X$ be a random variable taking values in $[0,\infty)$ and define $g_X(s) = \E[s^X]$, $s\in[0,1]$. Let $\beta\in[0,1]$ and $L$ be a slowly varying function at $\infty$. Then the following are equivalent:
 \begin{itemize}
  \item $1-g_X(s) \sim (1-s)^\beta L(1/(1-s))$, $\tas s\uparrow 1$.
  \item If $\beta\in[0,1)$,
  \[
   P(X > x) \sim \frac {L(x)}{x^\beta \Gamma(1-\beta)},\quad\tas x\to\infty.
  \]
  If $\beta = 1$, either $\E[\min(X,x)] \sim L(x)$ or $\E[X\Ind_{(X\le x)}] \sim L(x)$, as $x\to\infty$ (in which case both occur).
  \end{itemize}
\end{fact}

\section*{Acknowledgements} 
I thank Nicolas Curien for pointing out reference \cite{Janson2005} and anonymous referee for several useful suggestions that improved the presentation.

\bibliography{tree-indexed_RW_heavy_tails}
\end{document}